\pdfoutput=1
\documentclass[11pt, a4paper, english]{amsart}
\usepackage{amsmath} 
\usepackage{amsthm} 
\usepackage{amssymb} 

\usepackage{amscd} 
\usepackage{mathtools}
\usepackage[dvipsnames]{xcolor}

\usepackage{array} 
\usepackage[cal=boondoxo,scr=euler]{mathalfa}
\usepackage[backref=page,linktocpage]{hyperref} 
\usepackage{caption} %
\usepackage{graphics,graphicx} 
\usepackage{tikz,tikz-cd} 

\usetikzlibrary{graphs,graphs.standard,calc}

\usetikzlibrary{decorations.pathreplacing,}

\usepackage{enumerate} 
\usepackage{newtxtext,newtxmath} 
\DeclareMathAlphabet{\mathsf}{OT1}{\sfdefault}{m}{n}

\newcommand{\nocontentsline}[3]{}
\newcommand{\tocless}[2]{\bgroup\let\addcontentsline=\nocontentsline#1{#2}\egroup}

\usepackage[margin=1.25in]{geometry}
\usepackage{verbatim}
\usepackage{xcolor}

\usepackage{scalerel}

\makeatletter
\@namedef{subjclassname@2020}{\textup{2020} Mathematics Subject Classification}
\makeatother

\newcommand\SetOf[2]{\left\{#1\vphantom{#2}\ :\ #2\vphantom{#1}\right\}}

\DeclareMathAlphabet{\amathbb}{U}{bbold}{m}{n}

\hypersetup{
    colorlinks = true,
    linkbordercolor = {white},
    linkcolor = {BrickRed},
    anchorcolor = {black},
    citecolor = {BrickRed},
    filecolor = {cyan},
    menucolor = {red},
    runcolor = {cyan},
    urlcolor = {black}
}

\usetikzlibrary{automata}

\newtheoremstyle{teoremas}
{13pt}
{13pt}
{\slshape}
{}
{\bfseries}
{}
{.5em}
{}

\theoremstyle{teoremas}
\newtheorem{theorem}{Theorem}[section]
\newtheorem{corollary}[theorem]{Corollary}
\newtheorem{lemma}[theorem]{Lemma}
\newtheorem{conjecture}[theorem]{Conjecture}
\newtheorem{proposition}[theorem]{Proposition}
\newtheorem*{theorem-no-label}{Theorem}

\newtheoremstyle{definition}
{12pt}
{12pt}
{}
{}
{\bfseries}
{}
{.5em}
{}

\theoremstyle{definition}

\newtheorem{problem}[theorem]{Problem}
\newtheorem{question}[theorem]{Question}
\newtheorem{example}[theorem]{Example}
\newtheorem{remark}[theorem]{Remark}
\newtheorem*{notation}{Essential notation}

\DeclareMathOperator{\rk}{rk}

\newcommand{\M}{\mathsf{M}}
\newcommand{\N}{\mathsf{N}}
\newcommand{\U}{\mathsf{U}}

\renewcommand{\lambda}{\uplambda}
\renewcommand{\mu}{\upmu}

\newcommand{\RR}{\mathbb{R}}
\newcommand{\LL}{\mathsf{\Lambda}}
\newcommand{\LLF}[1]{\mathsf{\Lambda}_{k,n}^{#1}}

\newcommand{\rank}{\operatorname{rk}}
\newcommand{\xc}{\operatorname{xc}}

\AtBeginDocument{%
   \def\MR#1{}
}

\title{Face enumeration for split matroid polytopes}

\author[L.~Ferroni]{Luis Ferroni}

\address{
  School of Mathematics, Institute for Advanced Study, Princeton (NJ), United States
}
\email{ferroni@ias.edu}

\author[B.~Schr\"oter]{Benjamin Schr\"oter}

\address{
  Department of Mathematics, KTH Royal Institute of Technology, Stockholm, Sweden
}
\email{schrot@kth.se}

\thanks{LF is a member at the Institute for Advanced Study, funded by the Minerva Research Foundation, he was also partially supported by the Swedish Research Council grant 2018-03968. BS is supported by the Swedish Research Council grant 2022-04224.}

\keywords{Matroid polytopes, face enumeration, split matroids, paving matroids}

\subjclass[2020]{52B05, 52B40, 05B35}

\allowdisplaybreaks

\begin{document}

\begin{abstract}
    This paper initiates the explicit study of face numbers of matroid polytopes and their computation. We prove that, for the large class of split matroid polytopes, their face numbers depend solely on the number of cyclic flats of each rank and size, together with information on the modular pairs of cyclic flats. We provide a formula which allows us to calculate $f$-vectors without the need of taking convex hulls or computing face lattices. We discuss the particular cases of sparse paving matroids and rank two matroids, which are of independent interest due to their appearances in other combinatorial and geometric settings.
\end{abstract}

\maketitle

\section{Introduction}
\noindent 
To every matroid $\M$ one may associate its base polytope $\mathscr{P}(\M)$, carrying all the information of the matroid. Not only this polytope plays prominent roles in combinatorial optimization \cite{schrijver}, but also is of fundamental importance in tropical geometry \cite{MaclaganSturmfels,Joswig_book}, the theory of valuations \cite{derksen-fink,ardila-sanchez}, combinatorial Hodge theory, and the study of matroid invariants \cite{berget-eur-spink-tseng,eur-huh-larson,ferroni-schroter}. 

A question that arises naturally in the study of a convex polytope $\mathscr{P}\subseteq \mathbb{R}^n$ is how many faces of each dimension $\mathscr{P}$ has. The \textbf{\emph{f}-vector} of $\mathscr{P}$ is defined by
    \[ f(\mathscr{P}) := (f_0, f_1,\ldots, f_{d-1},f_d),\]
where $f_i := \#\{\text{$i$-dimensional faces of $\mathscr{P}$}\}$ for each $i\in \{0,\ldots, d\}$ and $d:=\dim \mathscr{P}$. In particular, the number of vertices of $\mathscr{P}$ is just $f_0$, the number of facets of $\mathscr{P}$ is $f_{d-1}$, and $f_d=1$. 

The difficulty of calculating the $f$-vector may vary drastically depending on the polytope~$\mathscr{P}$, on the properties it possesses, or on how it is described; for some concrete examples of the computation of $f$-vectors and certain related problems, see \cite{ziegler}. The family of possible vectors arising as the $f$-vector of a polytope is notoriously hard, and their classification is open in dimensions as low as four, see \cite{ziegler-2007}. Even in the case of $0/1$-polytopes of fixed dimension, although the set of possible $f$-vectors is finite, much remains to be discovered, see \cite{ziegler01}. 

In this article we will initiate the study of the explicit face enumeration of matroid polytopes, by focusing on the well-structured subclass of \textbf{split matroids}. 
The face structure of some special classes as positroids and lattice path matroids appeared in previous work, however without an explicit enumeration. 
The class of split matroids was introduced by Joswig and Schr\"oter in \cite{joswig-schroter} to study tropical linear spaces. They have received considerable attention in the past few years, including a forbidden minor characterization \cite{cameron-mayhew}, hypergraphs descriptions \cite{berczi}, Tutte polynomial inequalities \cite{ferroni-schroter0}, subdivisions and computation of valuations \cite{ferroni-schroter}, and conjectures about exchange properties on the bases \cite{berczi-schwarcz} which are related to White's conjecture. 

Even though the $f$-vector of the matroid base polytope constitutes an invariant of the matroid~$\M$ under isomorphisms, it is not valuative; see Example \ref{example:f-non-valuative} below. This makes its computation considerably subtler and difficult. In particular, for the case of split matroids we require a non-trivial modification of the machinery presented in \cite{ferroni-schroter}.

One important reason why split matroids deserve to be studied is that they encompass the classes of paving and copaving matroids. A long-standing conjecture often attributed to Crapo and Rota, appearing in print in \cite{mayhew}, predicts that asymptotically almost all matroids are sparse paving. There is some evidence supporting this assertion \cite{pendavingh-vanderpol}, but another intriguing conjecture affirms that even restricting to the enumeration of \emph{non sparse paving} matroids, the class of split matroids will continue to be predominant \cite[Conjecture~4.11]{ferroni-schroter}.

As of today, the problem of face enumeration of matroid polytopes has not been approached systematically in the literature, and to the best of our knowledge there are no prior articles addressing their computation. Nonetheless, there are some results that could be of interest in the study of $f$-vectors of certain classes of matroids. The computation of the cd-index of matroid polytopes of rank two appears in work by Kim \cite{kim}. In \cite{an-jung-kim} An, Jung and Kim investigated the lattice of faces of the base polytopes of lattice path matroids. In \cite{ardila-rincon-williams} Ardila, Rinc\'on and Williams approached the lattice of faces of positroids, whereas in \cite{oh-xiang} Oh and Xiang studied the facets of positroid polytopes. In \cite{grande-sanyal} Grande and Sanyal used the faces of matroid polytopes to characterize their $k$-levelness. In all of the aforementioned cases, although combinatorial descriptions and properties of the faces of the polytope are provided, an explicit enumeration of them does not seem direct or easy. In \cite{postnikov-reiner-williams} Postnikov, Reiner and Williams described the $h$-vector of simple generalized permutohedra; however, although the class of generalized permutohedra encompasses the family of matroid base polytopes, these fail to be simple when the rank or the corank are greater than one. 

In particular, perhaps as a reminiscence of the situation for polytopes in general (and even for $0/1$-polytopes), questions about properties of $f$-vectors of matroid polytopes are widely open. 

\subsection*{Main results}
As mentioned before, the fact that the face numbers are not valuations makes the computation of the $f$-vector of matroid polytopes a delicate task. In the case of split matroids, we need more data than just the number of cyclic flats of each rank and size. Some information on their pairwise intersection is necessary.

In order to express the $f$-vector of a polytope $\mathscr{P}$ in a more compact fashion, we will often refer to the \textbf{\emph{f}-polynomial}, which is defined via:
    \[ f_{\mathscr{P}}(t) := \sum_{i=0}^d f_i\cdot t^i.\]
Following the notation and terminology of \cite{ferroni-schroter}, whenever we have a matroid $\M$ of rank $k$ and cardinality $n$, we denote by $\lambda_{r,h}$ the number of stressed subsets of rank $r$ and size $h$ with non-empty cover that $\M$ has. For a connected split matroid, the number $\lambda_{r,h}$ counts also the number of proper non-empty cyclic flats of rank $r$ and size $h$. Although one of the main results of that article establishes that the numbers $\lambda_{r,h}$ are enough to compute any valuative invariant on $\M$, we need further data to compute the $f$-vector. 

For a matroid $\M$ as before, we will denote by $\mu_{\alpha,\beta,a,b}$ the number of \textbf{modular pairs} of cyclic flats $\{F_1,F_2\}$ such that $a = |F_1\smallsetminus F_2|$, $b = |F_2\smallsetminus F_1|$, $\alpha = \rk(F_1)-\rk(F_1\cap F_2)$, and $\beta = \rk(F_2)-\rk(F_1\cap F_2)$; see also equation~\eqref{eq:star-property-pairs-of-flats} below.

The following constitutes the main result of this article and is stated as Theorem~\ref{thm:main} further below. It tells us that the numbers $\mu_{\alpha,\beta,a,b}$ are the precise additional datum needed to perform the computation of the $f$-vector of a split matroid polytope. Moreover, the statement tells us concretely how to calculate the number of faces of given dimension.
\begin{theorem-no-label}
      Let $\M$ be a connected split matroid of rank $k$ on $n$ elements. The number of faces of its base polytope $\mathscr{P}(\M)$ is given by the polynomial
        \[
        f_{\mathscr{P}(\M)}(t) = 
        f_{\Delta_{k,n}}(t)-\sum_{r,h} \lambda_{r,h}\cdot u_{r,k,h,n}(t)-\sum_{\alpha,\beta,a,b} \mu_{\alpha,\beta,a,b}\cdot w_{\alpha,\beta,a,b}(t)
        \]
    where the first sum ranges over all values with $0<r<h<n$ and the second sum ranges over the values $0<\alpha<a$, $0<\beta<b$ for which either $a<b$ or $a=b$ and $\alpha\leq\beta$.  
\end{theorem-no-label}

In the above theorem, the expressions $u_{r,k,h,n}(t)$ and $w_{\alpha,\beta,a,b}(t)$ are polynomials which depend only on their subindices. We present in Propositions \ref{prop:nice-formula-w} and \ref{prop:nice-formula-u} explicit (but complicated) formulas for them which allow us or a computer to calculate the face numbers effortlessly. 
A formula for the $f$-vector of the hypersimplex $\Delta_{k,n}$ is also given explicitly in Example \ref{ex:hypersimplex}.
In particular, the entire calculation can be done without the necessity of building costly face lattices or computing convex hulls.

As direct but interesting application of our result, we provide closed expressions for the $f$-vector of sparse paving matroids, a class that made a prominent appearance in the theory of the extension complexity of independence polytopes \cite{rothvoss}. We also prove a fairly explicit formula for the $f$-vector of arbitrary rank two matroids, which is of independent interest due to the connection of these polytopes with edge polytopes of complete multipartite graphs \cite{ohsugi-hibi}.

\section{The number of faces of split matroids}\label{sec:two}

\subsection{The set up} Throughout this paper we will assume that the reader is familiar with the usual terminology and notation in matroid theory.
For the notions and machinery introduced very recently, in particular about \textbf{stressed subsets}, \textbf{relaxations}, and \textbf{Schubert elementary split matroids} we refer the reader to our previous article \cite[Sections~3--4]{ferroni-schroter}. Regarding \textbf{split matroids} and \textbf{elementary split matroids} the reader can consult the same article as well as \cite{joswig-schroter,berczi}.
However, basic knowledge on polytopes should be enough to follow the arguments and methods in this manuscript.

For a $d$-dimensional polytope $\mathscr{P}$ we denote by $f(\mathscr{P}):=(f_0,\ldots,f_d)$ its \emph{$f$-vector}, and by
    \[
    f_{\mathscr{P}}(t) := \sum_{i=0}^{d} f_i\ t^i
    \]
its \emph{$f$-polynomial}. In both cases, $f_i$ denotes the number of $i$-dimensional faces of $\mathscr{P}$. Notice that we omit the inclusion of $f_{-1}:=1$ for the empty set in both the $f$-vector and the $f$-polynomial, but we do include $f_d=1$ for the polytope itself. A basic property of $f$-polynomials that we will use without explicitly mentioning is the fact that it behaves multiplicatively under cartesian products of polytopes, i.e., $f_{\mathscr{P}_1\times \mathscr{P}_2}(t) = f_{\mathscr{P}_1}(t) \cdot f_{\mathscr{P}_2}(t)$. Another basic property that we use implicitly is that a matroid on a ground set of size $n$ with exactly $c$ connected components has a base polytope of dimension $\dim \mathscr{P}(\M) = n - c$ which is the product of the base polytopes of its $c$ direct summands.

\begin{notation} 
    Following \cite{ferroni-schroter}, whenever we have a matroid $\M$, unless specified otherwise, the rank of $\M$ is denoted by $k$ and the size of its ground set is denoted by $n$. We reserve the letters $r$ and $h$ for the rank and the size of stressed subsets that $\M$ may possess.
\end{notation}

Our aim is to find formulas for the number of faces of a matroid base polytope $\mathscr{P}(\M)$ whenever the matroid $\M$ is connected, i.e., $\mathscr{P}(\M)\subseteq \RR^n$ is of dimension $n-1$, and split. Note that under the assumption of connectedness the classes of split matroids and elementary split matroids coincide \cite[Theorem~11]{berczi}. Since the base polytope of a direct sum of matroids $\M_1\oplus \M_2$ is the cartesian product of $\mathscr{P}(\M_1)$ and $\mathscr{P}(\M_2)$, the $f$-vector of any disconnected split matroid can be recovered from the $f$-vector of the connected components, all of which are split as well.

The most basic example of a matroid polytope is the hypersimplex $\Delta_{k,n}$, the matroid base polytope of the uniform matroid $\U_{k,n}$ of rank $k$ on $n$ elements.

\begin{example} \label{ex:hypersimplex}
    The face enumeration of hypersimplices is encoded in the following $f$-polynomial:
    \begin{align*}
        f_{\mathscr{P}(\U_{k,n})}(t) = f_{\Delta_{k,n}}(t) 
    &= 
    \binom{n}{k} + \sum^{n-1}_{i=1} \binom{n}{i+1}\sum_{j = 1}^{i} \binom{n-i-1}{k-j}\cdot t^i \enspace . 
    \end{align*}
    This formula can be obtained by contracting and deleting the elements of $\U_{k,n}$.
    That is, by intersecting with hyperplanes of the form $x_i=0$ or $x_i=1$ (and forgetting the coordinate $i$) which leads to lower dimensional hypersimplices. For a detailed proof see for example \cite[Corollary~4]{hibi-li-ohsugi}.
\end{example}

The next example is similar to the one in \cite[Remark~5.9]{ferroni2} and gives a glimpse of the subtlety of the $f$-vector as a matroid invariant. In general, we see that the assignment $\M \mapsto f_{\mathscr{P}(\M)}(t)$ is an invariant of the matroid $\M$ that fails to be valuative. Hence its computation is a more delicate task, even for the case of paving or split matroids. In these cases, we cannot rely on the strength of \cite[Theorem~5.3]{ferroni-schroter} --- that result asserts that the evaluation of a valuative invariant on a split matroid~$\M$ can be achieved by knowing relatively little about the matroid $\M$, consisting in its rank $k$, its size $n$, and parameters $\lambda_{r,h}$ which denote the number of stressed subsets with non-empty cover of rank $r$ and size $h$ that $\M$ has.
If one is interested in knowing the $f$-vector of $\mathscr{P}(\M)$, the matroidal information we just mentioned is far from being enough. One of the main difficulties in order to carry out the enumeration of the faces of $\mathscr{P}(\M)$ consists of first identifying what matroid data we need in addition to the parameters mentioned before. 

\begin{example}\label{example:f-non-valuative}
   Consider the four matroids $\U_{3,6}$, $\M$, $\N_1$ and $\N_2$ with ground set $\{1,\ldots,6\}$ and rank three, whose families of bases are given as follows:
        \begin{align*}
            \mathscr{B}(\U_{3,6}) &:= \binom{[6]}{3},
            &\mathscr{B}(\N_1) &:= \binom{[6]}{3} \smallsetminus \{\{1,2,3\}, \{4,5,6\}\}\\
            \mathscr{B}(\M) &:= \binom{[6]}{3}
            \smallsetminus \{\{1,2,3\}\},
            &\mathscr{B}(\N_2) &:= \binom{[6]}{3} \smallsetminus \{\{1,2,3\}, \{3,4,5\}\}.
        \end{align*}
   The $f$-vectors of their base polytopes are respectively:
    \begin{align*}
        f(\mathscr{P}(\U_{3,6})) &= (20, 90, 120, 60, 12, 1),
        &f(\mathscr{P}(\N_1)) &= (18, 72, 102, 60, 14, 1),\\
        f(\mathscr{P}(\M)) &= (19, 81, 111, 60, 13, 1),
        &f(\mathscr{P}(\N_2)) &= (18, 72, 101, 59, 14, 1).
    \end{align*}
    All of these matroids are sparse paving. In particular, the two matroids $\N_1$ and $\N_2$ have, e.g., the same Tutte polynomial and the same Ehrhart polynomial --- in fact, via \cite[Corollary~5.4]{ferroni-schroter} any valuative invariant on these two matroids yields the same result. Yet, observe that their $f$-vectors differ in the third and the fourth entries.
\end{example}

\subsection{Schubert elementary split matroids and a technical lemma} 

By using \cite[Corollary~3.29]{ferroni-schroter}, we see that the intersection of the hypersimplex $\Delta_{k,n}$ with the half-space of a single split hyperplane leads to the polytope:
\begin{equation}\label{eq:cuspidal}
    \mathscr{P}(\LL_{k-r,k,n-h,n}) = \SetOf{x\in\Delta_{k,n}}{\sum_{i=1}^h x_i\leq r} \enspace ,
\end{equation}
for appropriate values $r$ and $h$. This is the base polytope of the Schubert elementary split matroid $\LL_{k-r,k,n-h,n}$, a matroid having exactly three cyclic flats: the empty set, the entire ground set, and one proper cyclic flat of size $h$ and rank $r$.
For the purposes of this paper, the reader may regard equation~\eqref{eq:cuspidal} as the definition of Schubert elementary split matroids. 

Let us introduce some notation that will help us formulate later our main results in a more compact fashion:
\begin{equation}\label{eq:definition-of-u}
    u_{r,k,h,n}(t) := f_{\Delta_{k,n}}(t)-f_{\mathscr{P}(\LL_{k-r,k,n-h,n})}(t). 
\end{equation}
The $i$-th coefficient of this polynomial is the difference between the number of $i$-dimensional faces of the hypersimplex $\Delta_{k,n}$ and the number of $i$-dimensional faces of the Schubert elementary split matroid $\LL_{k-r,k,n-h,n}$. A non-obvious property is that some of these coefficients may be negative while other are positive --- moreover, the actual sign of each individual coefficient a~priori depends on the four parameters $r,k,h,n$.

Before we go on, let us introduce a second polynomial, which will play an important role in the sequel. For fixed numbers $0< \alpha < a$ and $0<\beta < b$ let us define,
\begin{align*}
    w_{\alpha,\beta,a,b}(t)
    &:= f_{\Delta_{\alpha+\beta,a+b}}(t)-f_{\Delta_{\alpha,a}}(t)\cdot f_{\Delta_{\beta,b}}(t)-u_{\alpha,\alpha+\beta,a,a+b}(t)-u_{\beta,\alpha+\beta,b,a+b}(t)\\
    &\; = f_{\mathscr{P}(\LL_{\beta,\alpha+\beta,b,a+b})}(t) + f_{\mathscr{P}(\LL_{\alpha,\alpha+\beta,a,a+b})}(t) - f_{\Delta_{\alpha+\beta,a+b}}(t) - f_{\Delta_{\alpha,a}}(t)\cdot f_{\Delta_{\beta,b}}(t).
\end{align*}

Later, in Proposition~\ref{prop:nice-formula-w}, we provide a compact formula for the polynomial $w_{\alpha,\beta,a,b}(t)$ and a formula for the polynomial $u_{r,k,h,n}(t)$ in Proposition~\ref{prop:nice-formula-u} both of which can be used to compute these polynomials, bypassing the computation of $f$-vectors of Schubert elementary split matroids using the polytopes themselves.

\begin{remark}
    The intuition of why it is reasonable to consider and define the complicated expression above stems from \cite[Example~5.2]{ferroni-schroter}. As follows from the explanation there, if the assignment $\M\mapsto f_{\mathscr{P}(\M)}(t)$ were valuative, then the defining formula for $w_{\alpha,\beta,a,b}(t)$ would actually be identically zero. The polynomial $w_{\alpha,\beta,a,b}(t)$ quantifies (in a certain way) how far the map $\M\mapsto f_{\mathscr{P}(\M)}(t)$ is from being valuative.
\end{remark}

The following lemma is the key in the proof of our main result. Its proof constitutes arguably the most technical part of the paper.
In the Lemma and the text below we denote the Schubert elementary split matroid of rank $k$ on $n$ elements with proper cyclic flat $F$ by $\LLF{F}$. This matroid is isomorphic to the matroid $\LL_{k-\rank F,k,n-|F|,n}$.
Similarly, for a set $N$ of coordinates we use the notation $\Delta_{k,N}\cong\Delta_{k,|N|}$ to indicate the coordinates of the hypersimplex.

\begin{lemma}\label{lem:key} 
    Let $\N$ be a rank $k$ split matroid on $[n]$ whose cyclic flats are the four sets $\varnothing$, $F$, $G$, and $F\cup G=[n]$ of rank $0$, $r_F$, $r_G$, and $k$, respectively.
    Then
    \begin{align}\label{eq:twosplits}
        f_{\mathscr{P}(\N)}(t) 
        &= f_{\mathscr{P}(\LLF{F})}(t) + f_{\mathscr{P}(\LLF{G})}(t) - f_{\Delta_{k,n}}(t)
    \end{align}
    if $|F\cap G| + k < r_F+r_G$, and
    \begin{align}\label{eq:twosplits_w}
        f_{\mathscr{P}(\N)}(t)
        &= f_{\mathscr{P}(\LLF{F})}(t) + f_{\mathscr{P}(\LLF{G})}(t) - f_{\Delta_{k,n}}(t) - w_{r_F-c,r_G-c,|F|-c,|G|-c}(t)
    \end{align}
    where $c = |F\cap G|$ if $|F\cap G| + k = r_F+r_G$ otherwise.
\end{lemma}

\begin{proof} 
    Note that either the matroid $\N$ is connected or it is the direct sum $\U_{r_F,|F|}\oplus\U_{r_G,|G|}$.
    The matroid polytope of the latter is $\mathscr{P}(\U_{r_F,|F|}\oplus\U_{r_G,|G|}) = \Delta_{r_F,|F|}\times\Delta_{r_G,|G|}$ with $f$-polynomial $f_{\Delta_{r_F,|F|}}(t)\cdot f_{\Delta_{r_G,|G|}}(t)$. Moreover, in this case $|F\cap G| = 0$ and $k=r_F+r_G$. Thus  formula~\eqref{eq:twosplits_w} applies by definition of $w_{r_F,r_G,|F|,|G|}$. From now on, we assume that $\N$ is connected. 
    
    We compare the faces of the polytope $\mathscr{P}(\N)$ with those of $\Delta_{k,n}$.
    By \cite[Proposition 7]{joswig-schroter} every $d$-face of $\mathscr{P}(\N)=\mathscr{P}(\LLF{F})\cap\mathscr{P}(\LLF{G})$, and also $\mathscr{P}_F:=\mathscr{P}(\LLF{F})$ or $\mathscr{P}_G:=\mathscr{P}(\LLF{G})$ lies in a $d$-face of the hypersimplex $\Delta_{k,n}$ or in at least one of the hyperplanes $H_F=\{x\in\RR^n\,|\,\sum_{i\in F} x_i = r_F\}$ and $H_G=\{x\in\RR^n\,|\,\sum_{i\in G} x_i = r_G\}$. 
    Notice further that by \cite[Proposition 14]{joswig-schroter} we have that $|F\cap G| + k \leq r_F+r_G$.
    Thus the statement of the lemma indeed covers all possible cases.
    Furthermore, this shows that there is no point in the hypersimplex $\Delta_{k,n}$ violating both of the inequalities $\sum_{i\in F} x_i \leq r_F$ and $\sum_{i\in G} x_i \leq r_G$ simultaneously. We will use this crucial fact several times further below without recalling it explicitly. A first consequence is that the two hyperplanes $H_F$ and $H_G$ split a face of $\Delta_{k,n}$ into at most three maximal dimensional polytopes.
    
    Now, let us analyze the various cases.
    If $Q'$ is a $d$-dimensional face of $\mathscr{P}(\N)$ and not contained in a $d$-face of the hypersimplex $\Delta_{k,n}$, then either either $Q'$ lies in exactly one of the hyperplanes $H_F$ and $H_G$, or in both. If it is in exactly one of them, say $H_F$, then $Q'$ is a $d$-dimensional face of $\mathscr{P}_F$ and not a $d$-dimensional face of $\mathscr{P}_G$. If the $d$-dimensional face $Q'$ lies in both hyperplanes then $Q'$ lies in a $(d+1)$-dimensional face $Q$ of $\Delta_{k,n}$ which is subdivided into two parts. We discuss this situation in detail further below in this proof.
    
    Let us now assume $Q$ is a $d$-dimensional face of $\Delta_{k,n}$. We will go through the three possibilities of how many cells $Q$ is subdivided into by $H_F$ and $H_G$.
    If $Q$ remains undivided then either it is a face of $\mathscr{P}(\N)$ and also of both
    $\mathscr{P}_F$ and $\mathscr{P}_G$, or
    it is not a face of $\mathscr{P}(\N)$ and thus a face of exactly one of the two polyhedra 
    $\mathscr{P}_F$ and $\mathscr{P}_G$.
    In the case that the face $Q$ is subdivided into three parts of dimension $d$, it contributes a $d$-face to each of the four polytopes as well. It remains to analyze the situation in which $Q$ is subdivided into two polytopes $Q'$ and $Q''$. It cannot happen that both of these polytopes are faces of $\mathscr{P}(\N)$. If one of them, say $Q'$, is a face of $\mathscr{P}(\N)$ then $Q'$ is a face of one of the polytopes $\mathscr{P}_F$ and $\mathscr{P}_G$, while $Q$ is a face of the other. In total $Q$ and $Q'$ contribute a $d$-dimensional face to each of the four polytopes.
    If neither $Q'$ nor $Q''$ is a face of $\mathscr{P}(\N)$ then $Q'$ and $Q''$ meet in a $(d-1)$-face $\widetilde{Q}\subsetneq Q$ which is a face of the three polytopes 
    $\mathscr{P}(\N)$, $\mathscr{P}_F$ and $\mathscr{P}_G$. Moreover, the $d$-dimensional face $Q'$ is a face of one of the two polytopes $\mathscr{P}_F$ and $\mathscr{P}_G$ while $Q''$ is a $d$-dimensional face of the other one.
    The $(d-1)$-dimensional face $\widetilde{Q}$ is not a face of the hypersimplex $\Delta_{k,n}$ and sits in both hyperplanes $H_F$ and $H_G$.
    Every $(d-1)$-dimensional face of $\mathscr{P}(\N)$ that is not a face of $\Delta_{k,n}$ but lies in both $H_F$ and $H_G$ must be contained in a $d$-face of $\Delta_{k,n}$ as we mentioned at the beginning of this proof.
    Let us investigate the situation further.
    There must exist four pairwise disjoint sets $A,B,C,D$ that partition the ground set $[n]$ such that
    \[
    Q = \Delta_{|C|,C}\times \Delta_{k-|C|,A\cup B}\times\Delta_{0,|D|}
    \]
    where $C$ consists of all coordinates which are one and $D$ of those which are $0$, and
    \[
    \widetilde{Q} = \Delta_{|C|,C}\times \Delta_{\alpha,A}\times\Delta_{\beta,B}\times\Delta_{0,|D|}
    \]
    where $\alpha = r_F-|C\cap F|$, $\beta = r_G-|C\cap G|$, $A\subseteq F\smallsetminus G$ and  $B\subseteq G\smallsetminus F$.
    The inclusions follow (up to interchanging the roles of $A$ and $B$) from the fact that both hyperplanes $H_F$ and $H_G$ induce the same split of $\Delta_{\alpha+\beta,A\cup B}$.
    From this we get $k-|C| = \alpha + \beta =r_F-|F\cap C|+r_G-|G\cap C|$ and hence
    \begin{align*}
        |F\cap G|   \leq r_F+r_G-k
                    = |F\cap C| + |G\cap C| - |C|
                    = |F\cap G\cap C| - |C\smallsetminus(F\cup G)|\enspace .
    \end{align*}
    We obtain $F\cap G \subseteq C \subseteq F\cup G$, and also the equality $|F\cap G|+k = r_F+r_G$.
    Furthermore, every choice of such $C$, $A\subseteq F\smallsetminus G$ and $B\subseteq G\smallsetminus F$ leads to a face that sits in both hyperplanes and in a higher dimensional face of $\Delta_{k,n}$. Hence, for the inclusion-wise maximal such face we have to pick $C=F\cap G$ and $D=[n]\smallsetminus(F\cup G)$.
    We see that the expression
    \[
    w_{r_F-|C|,r_G-|C|,|F|-|C|,|G|-|C|}(t)
    \]
    counts exactly the two types of faces.
    In summary, our comparison verifies the formulas of equations \eqref{eq:twosplits} and \eqref{eq:twosplits_w} and thus the proof is complete.
\end{proof}

\subsection{Face counting of split matroids}

For a connected split matroid $\M$, let us define the following numbers that we have already mentioned in the introduction. The number of stressed subsets with non-empty cover having rank $r$ and size $h$, denoted by $\lambda_{r,h}$ --- recall that by \cite[Proposition~3.9]{ferroni-schroter}, in a connected split matroid this is the same as the number of proper non-empty cyclic flats of rank $r$ and size $h$. We also need the numbers $\mu_{\alpha,\beta,a,b}$ of (unordered) modular pairs $\{F_1,F_2\}$ of proper non-empty cyclic flats, i.e., $F_1$ and $F_2$ fulfilling the \emph{modularity property},
\begin{equation}\label{eq:star-property-pairs-of-flats}
    \rank(F_1) + \rank(F_2)= \rank(F_1\cap F_2) + \rank(F_1\cup F_2) , \tag{$\star$}
\end{equation}
where the indices denote the following quantities:
\begin{align*}
        a &=|F_1\smallsetminus F_2|,& \alpha &= \rank F_1 - \rank(F_1\cap F_2)\\
        b &=|F_2\smallsetminus F_1|,& \beta&=\rank F_2 - \rank(F_1\cap F_2) \enspace .
\end{align*}
Note that the set $F_1\cap F_2 \subsetneq F_1 \subsetneq [n]$ can not contain a circuit if $\M$ is a connected split matroid, thus it is an independent set, i.e., $\rank(F_1\cap F_2)=|F_1\cap F_2|$.

In terms of these numbers and variables, our main result is the following theorem.
\begin{theorem}\label{thm:main} 
    Let $\M$ be a connected split matroid of rank $k$ on $n$ elements.
    The number of faces of its base polytope $\mathscr{P}(\M)$ is given by the polynomials
    \begin{equation}\label{eq:main-result}
        f_{\mathscr{P}(\M)}(t) = 
        f_{\Delta_{k,n}}(t)-\sum_{r,h} \lambda_{r,h}\cdot u_{r,k,h,n}(t)-\sum_{\alpha,\beta,a,b} \mu_{\alpha,\beta,a,b}\cdot w_{\alpha,\beta,a,b}(t)\;,
    \end{equation}
    where the first sum ranges over all values with $0<r<h<n$ and the second sum ranges over the values $0<\alpha<a$, $0<\beta<b$ for which either $a<b$ or $a=b$ and $\alpha\leq\beta$.
\end{theorem}

Before moving towards the proof of this result, let us digress about the meaning of its statement. On one hand, note that the polynomials $f_{\Delta_{k,n}}(t)$, $u_{r,k,h,n}(t)$ and $w_{\alpha,\beta,a,b}(t)$ can be precomputed for all the occurring instances of the variables which appear as subindices. The first non-trivial fact that is deduced by our statement is that in addition to the parameters $\lambda_{r,h}$, which always appear in the computation of a valuative invariant, the precise additional matroidal datum needed to compute the $f$-vector consists of the numbers $\mu_{\alpha,\beta,a,b}$. Surprisingly, the last sum in equation~\eqref{eq:main-result} does not take into consideration the rank nor the size of the matroid $\M$ itself, only the intersection data for the modular pairs of flats. The second non-trivial fact is that it explains how to put together this information in order to effectively computing the $f$-vector of $\mathscr{P}(\M)$ for a split matroid, circumventing the necessity of constructing the polytope.

\begin{proof}[Proof of Theorem~\ref{thm:main}.]
    We follow the guidance of the proof of Lemma~\ref{lem:key} and compare the faces of $\mathscr{P}(\M)$ with those of $\Delta_{k,n}$ taking into account the polytopes $\mathscr{P}(\LLF{F_j})$ for the various cyclic flats $F_j$.
    We recall that there is no point in $\Delta_{k,n}$ that violates more than one of the inequalities $\sum_{i\in F_j} x_i \leq \rank F_j$.
    
    Now, if $Q$ is a $d$-face of $\Delta_{k,n}$ that contains a $d$-face of $\mathscr{P}(\M)$ then every polytope $\mathscr{P}(\LLF{F_i})$ has a $d$-face that is contained in $Q$. Hence, these faces do not contribute to any $u_{r,k,h,n}$ or $w_{\alpha,\beta,a,b}$.
    If $Q$ is a $d$-face of $\Delta_{k,n}$ whose interior has no point in common with $\mathscr{P}(\M)$, then all these points lie beyond exactly one of the split hyperplanes $\sum_{i\in F_j} x_i = \rank F_j$. Therefore, $Q$ contributes $1$ to $f_{\Delta_{k,n}}(t)$ and to $u_{\rank_{F_j},k,|F_j|,n}(t)$, but not to the other polynomials.
    It remains again to look at the $(d-1)$-faces $\widetilde{Q}$ of $\mathscr{P}(\M)$ that are contained in the interior of the $d$-face $Q$ of $\Delta_{k,n}$.
    Note that $\widetilde{Q}$ is a hyperplane in $Q$ that separates two of the vertices of $Q$. Each of the two vertices can only violate one of the inequalities, thus
    $\widetilde{Q}$ is a face of exactly two of the polytopes $\mathscr{P}(\LLF{F_j})$.
    
    Let us assume that $\overline{Q}$ is inclusion-wise maximal among all these $(d-1)$-faces of $\mathscr{P}(\M)$ in the relative interior of $d$-faces of $\Delta_{k,n}$.
    Then there must be pairwise disjoint sets $C$ and $D$ in the complement of $F\cup G$ such that $\overline{Q}$ is  a face of
    \[
        \Delta_{|C|,C}\times\Delta_{k-|C|,F\cup G}\times\Delta_{0,D}
    \]
    where $F$ and $G$ are the two cyclic flats for which $\overline{Q}$ is a face of $\LLF{F}$ and $\LLF{G}$.
    It follows that $k-|C|\leq \rank_{\M}(F\cup G)$ as $\overline{Q}$ is a face of $\mathscr{P}(\M)$.
    Furthermore, applying Lemma~\ref{lem:key} to the matroid $(\M\smallsetminus D) / C$, i.e., the contraction of $C$ and the deletion of $D$ in $\M$,  which is a split matroid of rank $k-|C|$ on $F\cup G$, yields that the hyperplanes of $F$ and $G$ coincide only if $k-|C|=\rank F + \rank G - |F\cap G|$.
    Moreover, because $\M$ is a connected split matroid we get
    \[
        |F\cap G| + \rank(F\cup G) \leq |F\cap G| + k \leq \rank F + \rank G \enspace .
    \]
    Thus $k-|C| = \rank F\cup G$ and $|F\cap G| = \rank F + \rank G - \rank(F\cup G)$. 
    Furthermore, Lemma~\ref{lem:key} shows also that the faces $\widetilde{Q}$ and $Q$ contribute a summand $t^{d-1}$ and $t^{d}$ to the corresponding $w_{\alpha,\beta,a,b}(t)$. This completes our proof.
\end{proof}

\begin{example}
    Let us take a look again at Example~\ref{example:f-non-valuative}. The matroids $\N_1$ and $\N_2$ are sparse paving, have rank $k=3$ and size $n=6$. In each case the proper non-empty cyclic flats are exactly the non-bases, yielding for both matroids $\lambda_{2,3}=2$. One can compute the corresponding polynomial, $u_{2,3,3,6}(t) = 1+9t+9t^2-t^4$. In $\N_1$, the intersection of the only pair of proper non-empty cyclic flats, $F_1=\{1,2,3\}$ and $F_2=\{4,5,6\}$, does not satisfy the property \eqref{eq:star-property-pairs-of-flats}, because $\rk(F_1\cap F_2)+\rk(F_1\cup F_2)=0+3$, whereas $\rk(F_1)+\rk(F_2) = 2 + 2 = 4$.
    
    For $\N_2$, the situation is different, as $F_1 = \{1,2,3\}$ and $F_2=\{3,4,5\}$ indeed satisfy \eqref{eq:star-property-pairs-of-flats}, and we have $a = |F_1\smallsetminus F_2|=2$, $b = |F_2\smallsetminus F_1|=2$, $\alpha = \rk(F_1)-|F_1\cap F_2| =  2 - 1 = 1$, and $\beta=\rk(F_2)-|F_1\cap F_2| = 2 - 1=1$, so that $\mu_{1,1,2,2}=1$ and we have to subratct $w_{1,1,2,2}(t) = t^2+t^3$ to obtain the correct $f$-polynomial, as we expected.
\end{example}

\subsection{Explicit formulas}
The polynomials $u_{r,k,h,n}(t)$ and $w_{\alpha,\beta,a,b}(t)$ in Theorem \ref{thm:main} are defined in terms of $f$-vectors of specific matroid polytopes. In this subsection we will present explicit descriptions for these polynomials, enabling us to do the face enumeration of a split matroid polytope, without any convex hull or face lattice computation.
To express the formulas in a compact form, we will make use of multinomial coefficients. Let $i,j,\ell$ be non negative integers, then
\[
    \binom{i+j+\ell}{i,j} := \binom{i+j+\ell}{i,j,\ell} = \frac{(i+j+\ell)!}{i!j!\ell!}\enspace .
\]

We begin with an explicit formula for the polynomials $w_{\alpha,\beta,a,b}(t)$.
\begin{proposition}\label{prop:nice-formula-w}
    For any $0 < \alpha < a$ and $0 < \beta < b$, the following formula holds:
    \[
    w_{\alpha,\beta,a,b}(t) =
    \sum_{i=0}^{a-\alpha-1} \sum_{j=0}^{\alpha-1}
    \sum_{i'=0}^{b-\beta-1} \sum_{j'=0}^{\beta-1} 
    \binom{a}{i,j}\binom{b}{i',j'} \cdot (1+t) \cdot t^{a+b-i-j-i'-j'-2} \enspace .
    \]
\end{proposition}

\begin{proof}
    We are revisiting the proof of Lemma \ref{lem:key}. We observe the following: 
    the faces counted by $w_{\alpha,\beta,a,b}(t)$ come in pairs, namely the faces of $\Delta_{\alpha,a}\times\Delta_{\beta,b}$ that do not lie in the boundary of $\Delta_{\alpha+\beta,a+b}$ and the faces of $\Delta_{k,n}$ one dimension higher.
    Furthermore, these faces are obtained by deleting and contracting elements such that the hyperplane $\sum_{\ell=1}^a x_\ell = \alpha$ still induces a split.
    Deleting $i$ and contracting $j$ of the first $a$ elements as well as deleting $i'$ and contracting $j'$ of the other $b$ elements then leads to such a pair of faces of dimension $a+b-i-j-i'-j'-2$ and $a+b-i-j-i'-j'-1$. Expressing the ways of choosing the elements with multinomial coefficients leads to the desired formula.
\end{proof}

For the polynomials $u_{r,k,h,n}(t)$ we provide the following formula.

\begin{proposition}\label{prop:nice-formula-u} 
For any $0<r<k<n$ and $r<h<n$ the following formula holds 
        \[
        u_{r,k,h,n}(t) = p_{r,k,h,n}(t)  - p'_{r,h}(t)\cdot p'_{k-r,n-h}(t)\cdot(1+t) + \sum_{i = r+1}^k \binom{h}{i}\binom{n-h}{k-i}
    \]
    where $p_{r,h}'(t) =f_{\Delta_{r,h}}(t)-\binom{h}{r}$ and
        \[
        p_{r,k,h,n}(t) = \sum_{j=0}^{h-r-1} \sum_{i=0}^{\min\{h-j,k-1\}} \sum_{\ell = 0}^{\min\{k-i-1,k-r-1\}} \sum_{m=0}^{\min\{n-h-\ell,n-k-j-1\}} 
        \binom{h}{i,j}\binom{n-h}{\ell,m} \cdot t^{n-1-i-j-\ell-m}.
    \]
\end{proposition}

\begin{proof}
    We begin by counting the faces of $\Delta_{k,n}$ that are not faces of $\mathscr{P}(\LL_{k-r,k,n-h,n})$.
    There are two types of such faces. Those that are entirely beyond the splitting hyperplane, and the ones that are separated by the hyperplane.
    Clearly there are $\sum_{i=r+1}^k\binom{h}{i}\binom{n-h}{k-i}$ vertices of $\Delta_{k,n}$ that are cut off by the hyperplane $\sum_{i=1}^h x_i = r$.
    The other faces are precisely those containing such a vertex. Hence we count faces for which more than $r$ of the first $h$ coordinates can be chosen to be one, where in total we have $k$ coordinates equal to one. We do so by picking coordinates whose value we fix, the remaining coordinates can either be zero or one, and the total number of ones is precisely $k$. Moreover, we do not want to fix all coordinates, i.e., we can fix at most $k-1$ coordinates to be equal to one and at most $n-k-1$ of them to be zero.
    
    Say we fix $i$ coordinates in $\{1,\ldots, h\}$ to be one and $j$ to be zero. Then $0 \leq j< h-r$ and $0\leq i < k$ as well as $i+j\leq h$, since there are only $h$ coordinates to select from. Similarly, we fix $\ell$ ones in the last $n-h$ coordinates and $m\geq 0$ zeros.
    Then $0\leq\ell < k-i$ and $\ell<k-r$ as it must be allowed to move another $1$ to the first coordinates to get there more than $r$ out of the $k$ ones. 
    Furthermore, $\ell+m\leq n-h$ as we select $\ell+m$ out of $\{h+1,\ldots,n\}$ and $m+j < n-k$ because otherwise we would have fixed all zeros and hence also the ones.
    Clearly, there are $\binom{h}{i,j}$ ways to select the $i$ and $j$ first coordinates and $\binom{h}{\ell,m}$ ways to fix the $\ell$ ones and $m$ zeros in the last coordinates. Every fixed coordinate reduces the dimension and hence counts $p_{r,k,h,n}(t)$ these type of faces that are not vertices.
    
    Now we count the $(d-1)$-faces of $\LL_{r,k,h,n}$ that split a $d$-face of $\Delta_{k,n}$. These are exactly the faces of $\Delta_{r,h}\times\Delta_{k-r,n-h}$ that contain a product of edges, i.e., are not of the from  
    $\Delta_{r,h}\times\{v\}$ or $\{w\}\times\Delta_{k-r,n-h}$ for vertices $v$ of $\Delta_{k-r,n-h}$ and $w$ of $\Delta_{k,h}$. Thus these faces are enumerated by 
    the polynomial $p'_{r,h}(t)\times p'_{k-r,n-h}(t)$.
    Each of these faces contributes $-t^{d-1}$ to $u_{r,k,h,n}(t)$. Furthermore, the $d$-faces of $\Delta_{k,n}$ that are separated by such a split do not contribute to $u_{r,k,h,n}(t)$. Therefore we subtract
    $p'_{r,h}(t)\cdot p'_{k-r,n-h}(t)\cdot(1+t)$ from $p_{r,k,h,n}(t)$ to obtain the polynomial $u_{r,k,h,n}(t)$.
\end{proof}

In the following we will specialize Theorem~\ref{thm:main} to some common and interesting classes of matroids. We begin with an example.

\begin{example}
    Let $\M$ be the projective geometry $\mathrm{PG}(2,3)$. This is a matroid on $n=13$ elements of rank $k=3$. It is split as it is in fact paving. This matroid has $13$ stressed hyperplanes, i.e., rank $k-1=2$ flats, all of which have cardinality $h=4$. In other words, we have $\lambda_{2,4} = 13$. In particular, to use the formula of Theorem~\ref{thm:main}, the polynomial
    \begin{align*}
    u_{2,3,4,13}(t) &= -\, t^{11} - 11\, t^{10} - 54\,t^{9} - 156\,t^{8} - 294\,t^{7} - 378\,t^{6}\\
    &\qquad - 336\,t^{5} - 195\,t^{4} + \, t^{3} + 166\,t^{2} + 114\,t + \,4
    \end{align*}
    is required. 
    Since projective geometries are modular matroids, any pair of distinct proper non-empty cyclic flats fulfills the property \eqref{eq:star-property-pairs-of-flats}. Also, every pair of them intersect in a single element. Moreover, for every pair of these cyclic flats we have $a = |F_1\smallsetminus F_2| = 3$, and by symmetry $b = |F_1\smallsetminus F_2| = 3$. Additionally, $\alpha=\rk(F_1) - |F_1\cap F_2| = 2 - 1 = 1$ and again by symmetry $\beta=\rk(F_2) - |F_1\cap F_2| = 1$. Therefore there is a single non-vanishing coefficient $\mu_{\alpha,\beta,a,b}$ which is
    \[
        \mu_{1,1,3,3} = \binom{13}{2} = 78 \enspace.
    \]
    It remains to compute:
    \[w_{1,1,3,3}(t) = t^5 + 7 t^4 + 15 t^3 + 9 t^2\enspace .\]
    Now applying Theorem~\ref{thm:main}, we obtain:
    \begin{align*}
    f_{\mathscr{P}(\mathrm{PG}(2,3))}(t) &= f_{\Delta_{3,13}}(t) - 13 \, u_{2,3,4,13}(t) - 78 \, w_{1,1,3,3}(t)\\
    &= t^{12} + 39 \, t^{11} + 455\, t^{10} + 2704\, t^{9} + 9893\, t^{8} + 24414\, t^{7} + 42666\, t^{6} + \\
    &\qquad 54054\, t^{5} + 49608\, t^{4} + 31707\, t^{3} + 12870\, t^{2} + 2808\, t + 234 \enspace .
    \end{align*}
\end{example}

\subsection{Face numbers of sparse paving matroids}
As mentioned in the introduction, it is conjectured that almost all matroids are sparse paving; see \cite{mayhew} for the details.
Furthermore, many famous examples of matroids fall into this class; notable examples are the Fano matroid, the V\'amos matroid, the complete graph on four vertices, and the duals of each of them. Sparse paving and paving matroids are split, so we can make use of our main result. For sparse paving matroids all the proper cyclic flats are circuit hyperplanes, i.e., of rank $r=k-1$ and size $h=k$. Using these parameters, Theorem~\ref{thm:main} simplifies to the following statement.

\begin{corollary}\label{coro:sparse-paving}
    Let $\M$ be a connected sparse paving matroid of rank $k$ on $n$ elements having exactly
    $\lambda$ circuit-hyperplanes, and let $\mu$ count the pair of circuit-hyperplanes which have $k-2$ elements in common. Then
    \[      f_{\mathscr{P}(\M)}(t) = 
            f_{\Delta_{k,n}}(t)-\lambda\cdot u(t)- \mu\cdot (t^2+t^3)
    \]
    where $u(t)$ is given by 
    \begin{align*}
        1-k\cdot(n-k)\cdot(t+1)+\left((n-k)\cdot (t+1)^{k+1}
        + k\cdot(t+1)^{n-k+1}-n\cdot(t+1)\right)\cdot t^{-1}\\
        +\left((t+1)^k+(t+1)^{n-k}-(t+1)^n-1\right)\cdot t^{-2}\enspace .
    \end{align*}
\end{corollary}
\begin{proof}
    By substituting $r$ by $k-1$ and $h$ with $k$ in Proposition~\ref{prop:nice-formula-u} we obtain
    \[
        u(t) = u_{k-1,k,k,n}(t) = p_{k-1,k,k,n}(t)-p'_{k-1,k}(t)\cdot p'_{1,n-k}(t)\cdot(1+t)+1
    \]
    where $p_{k-1,k,k,n}(t)$ is equal to
    \begin{align*}
    \sum_{i=0}^{k-1}\binom{k}{i}\sum_{m=0}^{n-k-1} \binom{n-k}{m}\cdot t^{n-k-m}\cdot t^{k-i}\cdot t^{-1}
    &= \left((t+1)^k-1\right)\cdot\left((t+1)^{n-k}-1\right)\cdot t^{-1} \enspace .
    \end{align*}
    Furthermore, the polytope $\Delta_{k-1,k}$ is an affine transformation of $\Delta_{1,k}$ hence
    \[
    p'_{k-1,k}(t) = p'_{1,k}(t) = 
    \sum_{i=2}^k \binom{k}{i}\cdot t^{i-1} = \left((t+1)^k-1-k\cdot t\right)\cdot t^{-1}\enspace .
    \]
    Using the same formula for $p'_{1,n-k}(t)$ leads to the desired formula.
\end{proof}

As was mentioned earlier, there are many famous matroids that are sparse paving. We saw four sparse paving matroids in the running Example \ref{example:f-non-valuative}. In the next example we take a look at a family of sparse paving matroids with maximum number of circuit-hyperplanes. This example demonstrates that one can derive a bit more from our calculations.
\begin{example} 
    Let $m\geq 3$. Consider the rank $k=4$ matroid $\M$ whose bases are the quadruples of affinely independent binary vectors of length $m$. This is a sparse paving matroid on $n=2^m$ elements with the maximum number $\lambda = \frac{1}{4}\binom{n}{3}$ of circuit-hyperplanes, as every three points define a circuit. A simple double counting argument reveals that the number of pairs of these circuit-hyperplanes
    that share two elements is $\mu=\frac{3n-12}{8}\binom{n}{3}$.
    The polytope $\mathscr{P}(\M)$ has $\frac{6n^2-57n+132}{8}\binom{n}{3}$ many square faces all of which are induced by split hyperplanes, and hence can be read off from the quadratic term of the product $p'_{3,4}(t)\cdot p'_{1,n-4}(t)$, and 
    $
    \left(4\,\binom{n-3}{2}-\binom{n-4}{4}+\binom{n}{4}-3\,n^2+17\,n-20\right)\frac{1}{4}\,\binom{n}{3}
    $ many triangular faces that are also faces of the hypersimplex $\Delta_{4,n}$, e.g., for $m=3$ the matroid is the binary affine cube and its polytope has $420$ square and $448$ triangular faces.
\end{example}

\subsection{Face numbers of rank two matroids}

A loopless matroid of rank two is trivially paving, and hence a split matroid. This allows us to use the strength of Theorem~\ref{thm:main} to compute their $f$-vectors. 

The key is the following elementary observation. The hyperplanes, i.e., the flats of rank one, of a loopless matroid of rank two form a partition of the ground set, and conversely, any partition of the ground set defines precisely a single rank two matroid having each part as a flat. The bases of the matroid are obtained by taking two elements of the ground set, not belonging to the same part. 

Base polytopes of matroids of rank two have made prominent appearances throughout algebraic combinatorics, under various guises. Notably, as is pointed out in \cite[Section~6.1]{ferroni-higashitani}, they coincide with edge polytopes of complete multipartite graphs --- we refer to that paper for the precise definition of edge polytopes and a short overview of them. In this vein, the work of Ohsugi and Hibi \cite{ohsugi-hibi} addresses the edge polytopes of complete multipartite graphs, motivated both from toric geometry and graph theory. In particular, the content of \cite[Theorem~2.5]{ohsugi-hibi} provides a formula for the $f$-vector of the edge polytope of an arbitrary complete multipartite graph, and thus for general rank two matroid polytopes. Let us point out that there appears to be an error in the formula as they stated it --- in particular within the quantity they denote by $\alpha_i$.
As an application of Theorem~\ref{thm:main} we can give another formula for the $f$-vector of these polytopes. 

\begin{corollary} \label{cor:rank2}
    Let $\M$ be a loopless matroid of rank two having $s$ hyperplanes with cardinalities $h_1,\ldots, h_s$. Then, the number of $i$-dimensional faces of $\mathscr{P}(\M)$ or, equivalently, the edge polytope of a complete multipartite graph with parts of sizes $h_1,\ldots,h_s$ is given by:
    \[f_i(\mathscr{P}(\M)) = \binom{n+1}{i+2}+(s-1)\binom{n}{i+2} - \sum_{j<\ell} \binom{h_j+h_\ell+1}{i+2} + (s-2)\sum_{j=1}^s \binom{h_j+1}{i+2} - \sum_{j=1}^s \binom{n-h_j}{i+2}.\]
\end{corollary}
\begin{proof}
    It is straightforward to check the formula for the case $s=2$, that is $\mathscr{P}(\M) = \Delta_{1,h_1}\times \Delta_{1,h_2}$.
    If $s>2$ then $\M$ is connected and we may apply Theorem~\ref{thm:main}. 
    Observe that all pairs of flats are trivially modular, and they are pairwise disjoint. Thus, the non-vanishing coefficients are $\mu_{1,1,h_j,h_\ell}$.
    In the following we omit parts of the long and tedious calculations but indicate the main steps.    
    For $k=2$ and $r=1$ using some change of summation and the Vandermonde identity we obtain
    \begin{align*}
        \sum_{j<l} w_{1,1,h_j,h_\ell}(t) &= 
        \sum_{i=2}^{n-1} \left( \sum_{j<\ell} \binom{h_j+h_\ell}{i+2} -\sum_{j=1}^s\left((s-1)\binom{h_j}{i+2}+(n-h_j)\binom{h_j}{i+1}\right)\right)\cdot t^i\, (1+t) \enspace .
    \end{align*}
    Similarly, we get this formula for $p'_{1,h_j}(t)\cdot p'_{1,n-h_j}(t)\cdot (1+t)$
    \begin{align*}
    \sum_{i=2}^{n-1} \left(\binom{n}{i+2}-\binom{n-h_j}{i+2}-h\binom{n-h_j}{i+1}-(n-h_j)\binom{h_j}{i+1}-\binom{h_j}{i+2}\right)\cdot t^i\,(1+t)
\end{align*}
and for $p_{1,2,h_j,n}(t)$ the expression
\begin{align*}
     \sum_{i=1}^{n-1}\left(\binom{n}{i+2}-\binom{n-h_j}{i+2}-h\binom{n-h_j}{i+1}\right)t^{i+1}
    +\sum_{i=1}^{n-1}\left(h_j\binom{n-1}{i+1}-h_j\binom{n-h_j}{i+1}\right)\cdot t^i
\end{align*}
where the two sums correspond to the indices $i=0$ and $i=1$ in the definition of $p_{r,k,h,n}(t)$.
This gives us $u_{1,2,h_j,n}(t) = p_{1,2,h_j,n}(t)-p'_{1,h_j}(t)p'_{1,n-h_j}(t)(t+1)+\binom{h_j}{2}$ and hence $\sum_{j=1}^s u_{1,2,h_j,n}$. We also need 
    \begin{align*}
        f_{\Delta_{2,n}}(t) &= \binom{n}{2} + \sum_{i=1}^{n-1} \binom{n}{i+1}\sum_{j=1}^i \binom{n-i-1}{2-j}\cdot t^i
        = \binom{n}{2} + 3\binom{n}{3} \, t +\sum_{i=2}^{n-1}\binom{n}{i+1}(n-i) \cdot t^i \enspace . 
    \end{align*}
    Now putting the pieces together, applying Pascal's identity, and cancelling many terms, one may obtain the desired formula of the statement (where additional effort is required for the constant, linear and quadratic term).
\end{proof}

\begin{remark}
    Kim described in \cite{kim} how the cd-indices of polytopes change when performing hyperplane splits. From this, he derived a formula for the cd-index of rank two matroid base polytopes. This formula depends on the cd-index of polytopes of rank two matroids with $1$, $2$ and $3$ hyperplanes. While there are explicit expressions for the cd-index when the matroid has $1$ or $2$ hyperplanes, the case with $3$ hyperplanes remains unsolved. We point out that the cd-index contains more information than the $f$-vector --- however, recovering the $f$-vector from the cd-index is often a very laborious task.
\end{remark}

\section{Final remarks and open problems}
\noindent 
Steinitz characterized $f$-vectors of $3$-polytopes. Since then, the $f$-vectors of $4$-polytopes have been intensively studied. In spite of the great interest from mathematicians, only little is known about the face numbers of higher dimensional polytopes \cite{ziegler-2007} or $0/1$-polytopes \cite{ziegler01}. The contribution of this paper is an explicit formula for $f$-vectors of split matroids. However, many questions and problems regarding $f$-vectors of $0/1$-polytopes, or even the class of matroid polytopes itself, remain broadly open. This last section aims to propose a few problems and questions in this underexplored area of discrete geometry.

\subsection{Explicit formulas for other polytopes}

Besides the base polytope studied in this paper, another famous polytope that one may associate to a matroid is the so-called \emph{independence polytope}. This polytope is defined as the convex hull of the indicator vectors of all the independent sets, i.e., subsets of bases, of the matroid, and it contains the base polytope as a facet.

A natural challenge that we raise is to provide a good way of computing the $f$-vectors of these polytopes.

\begin{problem}
    Find a formula for the $f$-vector of the independent set polytope for (connected) split matroids.
\end{problem}

We speculate that one can approach this problem by using modifications of the ideas and the techniques that we presented in this article.

On the other hand, it is also natural to ask about the problem of finding a formula for $f$-vectors of matroid base polytopes or independence polytopes for \emph{arbitrary matroids}. These two problems appear to be considerably more difficult. Nonetheless, we pose the following broad question.

\begin{question}
    Can the approach of splitting polytopes, or treating the $f$-vector as a valuation with an error term, be used to obtain an explicit formula for the $f$-vector of an arbitrary matroid polytope? Independently of the details of the computation, what is the precise matroidal data that one needs in order to recover the $f$-vector?
\end{question}

Note that all matroid polytopes are cut out of the hypersimplex by intersecting it with split hyperplanes --- or, more precisely, with half-spaces whose boundary is a split hyperplane --- which are weakly compatible, i.e., that may intersect in the interior of the hypersimplex.

Also, every matroid base polytopes can be obtained by slicing pieces off of a unit cube. Unit cubes themselves possess many split hyperplanes; see \cite{JoswigHerrmann} for more details. Therefore, it seems reasonable to attempt to generalize the techniques and ideas of the present paper beyond the class of matroid polytopes, in order to include other classes of $0/1$-polytopes.

\begin{problem}
    Describe the $f$-vectors of $0/1$-polytopes in terms of their supporting split hyperplanes. 
\end{problem}

\subsection{The shape of \emph{f}-vectors of matroid polytopes}

A recent trend in matroid theory is that of proving unimodal and log-concave inequalities for various vectors of numbers associated to matroids. A finite sequence of numbers $(a_0,\ldots,a_n)$ is said to be \emph{unimodal} if there exists some index $0\leq j\leq n$ with the property that 
\[a_0\leq \cdots \leq a_{j-1}\leq a_j\geq a_{j+1}\geq \cdots \geq a_n.\] 
If all the $a_i$'s are positive, a stronger condition is that of \emph{log-concavity}, which asserts that for each index $1\leq j\leq n-1$ the inequalities $a_j^2\geq a_{j-1}a_{j+1}$ hold. 

There are two objects that, in other contexts, can be referred to as ``the $f$-vector of a matroid,'' though they hold no direct relation with the $f$-vector we studied in this paper. The first of them comes from considering a matroid as a pure simplicial complex with the property that each vertex-induced subcomplex is shellable \cite{bjorner}. In this case, one can define the $f$-vector of a matroid to be the $f$-vector of this simplicial complex. This has been object of intensive research and several open problems and conjectures exist in the literature regarding this $f$-vector, for instance Stanley's conjecture asserting that the corresponding $h$-vector is a pure $O$-sequence (see \cite{stanley-green}). In \cite{adiprasito-huh-katz} these $f$-vectors are proved to be log-concave.
The second object that sometimes is referred to as ``the $f$-vector of a matroid'' is the vector having as coordinates the number of flats of the matroid of each rank, i.e., the Whitney numbers of the second kind of $\M$. This object has also received some attention in the recent years; one of the main open problems regarding this $f$-vector was the so-called ``Top-Heavy Conjecture,'' which was settled in \cite{braden-huh-matherne-proudfoot-wang}. It remains as an open problem to prove or disprove the unimodality of the Whitney numbers of the second kind. 

In a similar vein, it is quite inviting to ask the following question.

\begin{question}
    Are the $f$-vectors of matroid base polytopes unimodal, or even log-concave?
\end{question}

It is known that there are simplicial polytopes having a non-unimodal $f$-vector; see \cite[Chapter~8.6]{ziegler}. However, let us point out that, within the existing literature, we were not able to find any examples of non-unimodal $f$-vectors for the general class of \emph{all} $0/1$-polytopes. We have been able to verify the log-concavity of the $f$-vectors of the following classes of matroids, in some cases relying critically on the results of this paper:
\begin{itemize}
    \item All matroids on a ground set of size at most $9$.
    \item Split matroids on a ground set of size at most $12$.
    \item Sparse paving matroids on a ground set of size at most $40$.
    \item Split matroids with four cyclic flats as in Lemma~\ref{lem:key} of size at most $50$.
    \item Schubert elementary split matroids on a ground set of size at most $100$.
    \item Lattice path matroids on a ground set of size at most $13$.
    \item Rank two matroids on a ground set of size at most $60$.
\end{itemize}

Matroid base polytopes have many remarkable properties. One that is particularly relevant is that their vertex-edge graph, i.e., the $1$-skeleton of the polytope, determines the entire polytope up to a rigid transformation; see \cite{holzmann-norton-tobey,SchroeterPineda} for more precise statements. Similarly,
the data of the $2$-skeleton of a matroidal subdivision describes the subdivision completely; this technique has been used in several articles, see for example \cite{HerrmannJoswigSpeyer}. It is tempting to ask whether the enumerative information encoded in the first few entries of the $f$-vector of a matroid base polytope is already sufficient to derive the remaining entries. More precisely, we ask for the following constant.

\begin{problem}
    Is it true that there exists a number $c$ such that the first $c$ entries of the $f$-vector of some matroid $\M$, i.e., $f_0, \ldots, f_{c-1}$, are enough to determine the complete $f$-vector of $\M$? If it is true, what is the smallest such $c$?
\end{problem}

One may consider various versions of the above problem, where the size or the rank of the matroid are part of the input or not. Further slight modifications consist of restricting to only connected matroids, or to matroids belonging to a specific class.
Our computer experiments indicate that $c=5$ suffices for all connected matroids on up to nine elements. This number cannot be smaller as we found two matroids of rank four on nine elements that agree on the sequence $f_0$, $f_1$, $f_2$, $f_3$ but not on the number of $4$-faces $f_4$. Their $f$-vectors are:
\[
(96, 753, 2110, 2934, 2262, 993, 240, 28, 1)
\text{ and } (96, 753, 2110, 2934, 2261, 992, 240, 28, 1).
\]

On the other hand, as a consequence of Corollary~\ref{coro:sparse-paving}, if one restricts to all sparse paving matroids of rank $k$ on $n$ elements, knowing the number of vertices $f_0$ and $2$-dimensional faces $f_2$ suffices to recover $\lambda$ and $\mu$, and hence the entire $f$-vector. Thus, in this version of the problem the answer is $c=3$.

\subsection{A digression on the extension complexity of split matroids} 

A motivation to study the face enumeration of matroid polytopes stems from the work of Rothvoss \cite{rothvoss},
Conforti, Kaibel, Walter and Weltge \cite{ConfortiKaibelWalterWeltge}, and Kaibel, Lee, Walter and Weltge \cite{KaibelLeeWalterWeltge}
on the extension complexity of matroid polytopes (see also the corrigendum \cite{KaibelLeeWalterWeltge_Correction}, and the article by Aprile and Fiorini \cite{aprile-fiorini}). The special case of hypersimplices is discussed in work by Grande, Padrol and Sanyal \cite{GrandePadrolSanyal}, and the case of $2$-level matroids is the main focus of \cite{aprile-cevallos-faenza,aprile-conforti-fiorini-faenza-huynh-macchia}.

Given a lattice polytope $\mathscr{P}\subseteq\mathbb{R}^n$, an \emph{extended formulation} of $\mathscr{P}$ is another lattice polytope $\mathscr{Q}\subseteq\mathbb{R}^m$ together with a projection map $\pi :\mathbb{R}^m \to \mathbb{R}^n$ which projects $\mathscr{Q}$ onto $\mathscr{P}$. The \emph{complexity} of an extended formulation is the number of facets of the polytope~$\mathscr{Q}$. The \emph{extension complexity} of $\mathscr{P}$, denoted $\xc(\mathscr{P})$, is the minimum complexity of an extended formulation of $\mathscr{P}$. 

One of the main tools for finding lower bounds on the extension complexity are the rectangular covering numbers, but these numbers grow at most quadratically in the size of the ground set \cite[Proposition~2]{KaibelLeeWalterWeltge}.
Furthermore, the extension complexity of regular matroids is polynomial \cite{aprile-fiorini}. The extension complexity of matroid polytopes is also related to the ``hitting number'' of the base polytope, see \cite{aprile}.

Rothvoss \cite[Corollary~6]{rothvoss} proved\footnote{To be precise, Rothvoss proved that the extension complexity of the \emph{independence polytope} of some matroid is exponential, but an elementary reasoning shows that this is equivalent to an analogous statement for the base polytope. See for example the short explanation in \cite[p.\ 1]{aprile-fiorini}.} that for all $n$ there exists a matroid $\M$ on $n$ elements whose base polytope has extension complexity $\xc(\mathscr{P}(\M)) \in \Omega\left(\frac{2^{n/2}}{n^{5/4} \sqrt{\log(2n)}}\right)$. Moreover, Rothvoss' proof is non-constructive and relies only on an enumerative result of matroids, that therefore guarantees that whatever these examples are, they must belong to the class of sparse paving matroids, and are therefore split matroids. Thus it remains a notorious open problem to find an explicit family of matroids having exponential extension complexity. In fact, having one would yield an explicit infinite family of Boolean functions requiring superlogarithmic depth
circuits, according to an observation attributed to G\"o\"os in \cite[Section~8]{aprile-fiorini} and \cite{Matroid-union-post}; see also the relevant \cite{goos-jain-watson}.

Although we cannot compute explicitly the extension complexity of split matroid polytopes, we conjecture that the following natural family of sparse paving matroids might have the desired exponential extension complexity.

\begin{conjecture}\label{conj:extension-complexity}
    For each positive integer $n$, let us denote by $\mathsf{S}_n$ the sparse paving matroid on $[n]$ of rank $\lfloor\frac{n}{2}\rfloor$, having the maximal possible number of circuit-hyperplanes, and whose set of bases is lexicographically minimal\footnote{In the sense that if we write each individual basis $B\in\mathscr{B}$ with their elements in increasing order, and sort the set $\mathscr{B}$ lexicographically, then $\mathscr{B}$ is minimal.}. Then 
    \[\xc(\mathscr{P}(\mathsf{S}_n)) \in \Omega\left(\frac{2^{n/2}}{n^{5/4} \sqrt{\log(2n)}}\right).\]
\end{conjecture}

Observe that the explicit determination of the matroid $\mathsf{S}_n$ is yet another open problem, related to the construction of binary codes with constant weight and Hamming distance $4$, as well as stable subsets of Johnson graphs. In particular, we point out that it remains an open problem to determine the matroid $\mathsf{S}_{20}$ --- in fact, its number of circuit-hyperplanes seems to be unknown, though it is between $13452$ and $16652$ (see~\cite[Table~I]{agrell}). 

Although we cannot prove this conjecture, at least we can prove that the number of facets of $\mathscr{P}(\mathsf{S}_n)$ is indeed exponential.

\begin{proposition}
    Let $\mathsf{S}_n$ be the matroid described in the prior statement. The number of facets of $\mathscr{P}(\mathsf{S}_n)$ satisfies:
        \[ f_{n-2}(\mathsf{S}_{n}) \geq 2n + \frac{1}{n}\binom{n}{\lfloor\frac{n}{2}\rfloor} \sim c \,\frac{2^n}{n^{3/2}}. \]
\end{proposition}

\begin{proof}
    By Corollary~\ref{coro:sparse-paving} the number of facets of $\mathsf{S}_n$ is $2n+\lambda$ (whenever $n>4$) where $\lambda$ is the number of circuit-hyperplanes of $\mathsf{S}_n$. As follows from a result of Graham and Sloane in \cite{graham-sloane}, the maximal possible value of $\lambda$ is at least $\frac{1}{n}\binom{2n}{n}$; see for example \cite[Lemma~4.14]{ferroni-schroter}.
\end{proof}

\begin{remark}
    Even though the number of facets of a matroid polytope can be exponential, for arbitrary $0/1$-polytopes in $\mathbb{R}^n$ it is known that the number of facets can be larger than $\left(\frac{cn}{\log n}\right)^{n/4}$, via a random construction \cite{barany-por}.
\end{remark}

\bibliographystyle{amsalpha0}
\bibliography{bibliography}

\end{document}